\documentclass{amsart}
\usepackage{amsmath}
\usepackage{amsfonts}
\usepackage{amssymb}
\usepackage{color}
\usepackage{graphics,graphicx}
\usepackage{hyperref}
\usepackage{amssymb,latexsym,amsthm,xypic}

\newtheorem{thm}{Theorem}[section]
\newtheorem{cor}[thm]{Corollary}
\newtheorem{lem}[thm]{Lemma}

\theoremstyle{definition}

\theoremstyle{remark}

\newcommand{\be}{\begin{equation}}
\newcommand{\ee}{\end{equation}}
\newcommand{\bea}{\begin{eqnarray}}
\newcommand{\eea}{\end{eqnarray}}
\newcommand{\ben}{\begin{eqnarray*}}
\newcommand{\een}{\end{eqnarray*}}
\newcommand{\bt}{\begin{split}}
\newcommand{\et}{\end{split}}
\newcommand{\bet}{\begin{equation}}
\newcommand{\mc}{\mathbb{C}}

\newcommand{\ra}{\rightarrow}

\begin{document}

\title[]{Fridman's invariant, squeezing functions, and exhausting domains}
\date{}
\author[F. Deng, X. Zhang]{Fusheng Deng, Xujun Zhang}
\address{Fusheng Deng: \ School of Mathematical Sciences, University of Chinese Academy of Sciences\\ Beijing 100049, China}
\email{fshdeng@ucas.ac.cn}
\address{Xujun Zhang: \ School of Mathematical Sciences, University of Chinese Academy of Sciences\\ Beijing 100049, China}
\email{zhangxujunmath@icloud.com}

\begin{abstract}
We show that if a bounded domain $\Omega$ is exhausted by a bounded strictly pseudoconvex domain $D$ with $C^2$ boundary,
then $\Omega$ is holomorphically equivalent to $D$ or the unit ball,
and show that a bounded domain has to be holomorphically equivalent to the unit ball if its Fridman's invariant
has certain growth condition near the boundary.
\end{abstract}

\maketitle
\section{Introduction}
In \cite{Fri83}, Fridman introduces a new invariant for bounded domains as follows.
Let $D\subset\mc^n$ be a bounded domain and $z\in D$.
Let $f:B^n\ra D$ be a holomorphic injective map such that $f(0)=z$,
where $B^n$ is the unit ball in $\mc^n$.
We define $q_{D,f}(z)$ to be the supremum of all $r>0$ such that $B^K_D(z,r)\subset f(B^n)$,
where $B^K_D(z,r)$ is the open ball in $D$ centered at $z$ with radius $r$ with respect to the Kobayashi metric on $D$.
Then Firdman's invariant is defined to be
$$h_D(z)=\inf_f \frac{1}{q_{D,f}(z)},$$
where the supremum is taken to be all injective holomorphic maps $f:B^n\ra D$ with $f(z)=0$.

Form the definition, it is clear that $h_D$ is invariant under biholomorphic transformations.
Some basic results about the invariant $h_D$ were proved in \cite{Fri83}.
The most important one says that for a bounded strictly pseudoconvex domain $D$ with $C^3$ boundary,
$h_D(z)$ tends to $0$ as $z$ goes to the boundary.

In \cite{Deng-Guan-Zhang12}, Deng-Guan-Zhang introduce another invariant of bounded domains, called squeezing function, as follows.
Let $D$ be a bounded domain in $\mathbb{C}^{n}$ and $z\in D$.
For an injective holomorphic map $f:D \rightarrow B^{n}$ with $f(z)=0$,
we define
$$s_{D}(p,f)=sup \left\{r| B(0,r) \subset f(D) \right\},$$
and define
$$s_D(z)=\underset{f}{sup}\left\{s_{D}(p,f) \right\},$$
where the supremum is taken over all injective holomorphic maps $f:D \rightarrow B^{n}$ with $f(p)=0$, and $B^{n}(0,r)$
is the Euclidean ball in $\mathbb{C}^{n}$ with center 0 and radius $r$.
As $z$ varies, we get a function $s_D$ on $D$, which is called the squeezing function of $D$.

By definition, a bounded domain is called \emph{homogenous regular}
if its squeezing function has positive lower bound.
The notion of homogenous regular domains was introduced and studied in \cite{Liu-Sun-Yau04, Liu-Sun-Yau05},
and was called uniform squeezing domains and systematically studied in \cite{Yeung09}.

In addition to the works in \cite{Liu-Sun-Yau04}\cite{Liu-Sun-Yau05}\cite{Yeung09},
the new motivation for introducing the concept of squeezing function
is boundary estimate.
The study of boundary estimate of squeezing functions was initiated in \cite{Deng-Guan-Zhang12}
and was further developed in \cite{Diedrich-Fornaess-Wold13} \cite{Deng-Guan-Zhang16}.
In recent years, boundary estimate of squeezing functions and their applications
in different settings were extensively studied by different authors
(see e.g. \cite{Diedrich-Fornaess-Wold16}\cite{Fornaess-Wold15}\cite{Diederich-Fornaess15}\cite{Fornaess-Shcherbina}
\cite{Fornaess-Rong16}\cite{Zimmer16}\cite{Nikolov-Andreev16}\cite{Fornaess-Wold16-preprint}\cite{Joo-Kim16}
\cite{Zhang15}\cite{Zimmer17}\cite{Nikolov17}\cite{Bracci-Fornaess-Wold17}\cite{Arosio-Fornaess-Shcherbina-Wold17}\cite{Zimmer18}\cite{Nik-Try18}).

One motivation for Fidman to propose the invariant $h_D$ defined above
is to study exhausting domains.
Let $D, \Omega$ be domains in $\mc^n$.
Following Fridman, we say that $\Omega$ can be exhausted by $D$ or $D$ can exhaust $\Omega$ if for any compact subset $K$ of $\Omega$
there exists an injective holomorphic map $f:D\ra\Omega$ such that $K\subset f(D)$.

It is proved by Fridman that there are domains in $\mc^m$
which can exhaust all domains \cite{Fri83}.
It is also observed by Forn{\ae}ss that the unit ball $B^n$
can exhaust many domains which are not biholomorphic to each other \cite{For04}.
However, if we restrict on bounded domains, things become very different.
For example, it is easy to prove that if a bounded domain $\Omega$ can be exhausted
by a homogenous domain $D$, then $\Omega$ must be holomorphic equivalent to $D$ (see \cite{Fri83}).

In \cite{Fri83}, Fridman shows that if a bounded domain $\Omega$ can be exhausted by
a bounded strictly pseudoconvex domain $D$ with $C^3$ boundary,
then $\Omega$ must be biholomorphic to $D$ or the unit ball $B^n$.
The proof is based on the boundary estimate of $h_D$ for strictly pseudoconvex domains with $C^3$ bounday.
In this note, applying the estimate of squeezing functions in \cite{Deng-Guan-Zhang16}
and based on Firdman's idea, we show that the same result still holds if the boundary regularity of $D$ is reduced to $C^2$.

\begin{thm}\label{thm-intr:exhauting domain}
Let $D$ be a bounded strictly pseudoconvex domain in $\mc^n$ with $C^2$ boundary.
If a bounded domain $\Omega\subset\mc^n$ can be exhausted by $D$,
then $\Omega$ must be biholomorphic to $D$ or the unit ball $B^n$.
\end{thm}

Another result in this note about exhausting domains is the following

\begin{thm}\label{thm-intr:exhausted by USD}
If a bounded domain $\Omega$ can be exhausted by a homogenous regular domain,
then $\Omega$ is homogenous regular.
\end{thm}

Another purpose of the present note is to give an boundary estimate of $h_D$.
By comparing with squeezing functions, we consider a variant $e_D$ of $h_D$ which is defined by
$$h^{-1}_D(z)=\log\frac{1+e_D(z)}{1-e_D(z)}.$$
Then $e_D(z)\in (0,1]$ and is set to be $1$ if $h_D(z)=0$.

\begin{thm}\label{thm-intr:ball charac}
Let $D\subset\mc^n$ be a bounded domain and $p\in \partial D$ be a $C^2$ boundary point of $D$.
If there is a sequence $z_j\in D\ (j\geq 1)$ converging to $p$ and a sequence of positive numbers
$\epsilon_j\ (j\geq 1)$ converges to $0$ such that $e_D(z_j)>1-\epsilon_j \delta(z_j)$ for all $j$,
then $D$ is biholomorphic to the unit ball,
where $\delta(z)$ denotes the distance between $z$ and $\partial D$.
\end{thm}

By the decreasing property of the Kobayashi metric,
it is obvious that $s_D\leq e_D$ for all bounded domains (see \cite{Nik-Ver18}).
So as a corollary of Theorem \ref{thm-intr:exhauting domain},
we get the main result of Diederich-Forn{\ae}ss-Wold in \cite{Diedrich-Fornaess-Wold13}
which says that, for a bounded strictly pseudocovnex domain $D$ with $C^2$ boundary which is not biholomorphic to the unit ball,
$s_D(z)\leq C\delta(z)$ for some constant $C$.

On the other hand, it is proved by Forn{\ae}ss-Wold that
for a bounded strictly pseudoconvex domain $D$ with $C^4$ boundary,
the estimate
$$s_D(z)\geq 1-C\delta(z), z\in D$$
holds for some constant $C>0$ \cite{Fornaess-Wold15}.
Combing this result with Theorem \ref{thm-intr:ball charac},
we obtain the following

\begin{cor}
Let $D$ be a bounded strictly pseudoconvex domain with $C^4$ boundary.
Then there exists a constant $C>0$ such that
$$1-C^{-1}\delta(z)\leq e_D(z)\leq 1-C\delta(z), z\in D.$$
\end{cor}

It seems that for a strictly pseudoconvex domain $D$,
the boundary behaviour of $e_D(z)$ and $s_D(z)$ are very similar.
So we conjecture that the following comparison
$$\lim_{z\ra \partial D}\frac{1-e_D(z)}{1-s_D(z)}=1$$
holds.\\\\

During the preparation of this note,
a paper of Nikolov and Verma \cite{Nik-Ver18} discussing some related topics appears in arXiv.
After the first version of this note appeared in arXiv, 
Professor Nikolai Nikolov pointed out to the authors that Theorem 1.3 in this note 
is contained in Proposition 8 in \cite{Nik-Ver18}.
Proposition 8 in \cite{Nik-Ver18} is stated for squeezing functions and itself can not contain Theorem \ref{thm-intr:ball charac}.
But after looking into the details of its proof, 
we realized that the same argument can give a proof of Theorem \ref{thm-intr:ball charac}.

\subsection*{Acknowledgements}
The first author thanks Professor John Erik Forn{\ae}ss for discussions.
The authors are partially supported by NSFC grants.

\section{Exhausting a domain by strictly pseudoconvex domains}
The aim of this section if to prove Theorem \ref{thm-intr:exhauting domain}.

We first recall some results about squeezing functions from \cite{Deng-Guan-Zhang12}\cite{Deng-Guan-Zhang16}.

\begin{lem}[\cite{Deng-Guan-Zhang12}]\label{lem:cha of ball by squeezing}
Let $D\subset\mc^n$ be a bounded domain.
If there is a point $z\in D$ such that $s_D(z)=1$,
then $D$ is biholomorphic to the unit ball $\mathbb B^n$.
\end{lem}

\begin{lem}[\cite{Deng-Guan-Zhang16}]\label{lem:stability squeezing}
Let $D, D_k\subset\mc^n$ ($k\geq 1$) are bounded domains with $D_k\subset D$
such that, for any compact set $K\subset D$, there is $N>0$ such that $K\subset D_k$ for all $k>N$.
Then for any $z\in D$, we have $\lim_{k\ra\infty}s_{D_k}(z)=s_D(z)$.
\end{lem}

In \cite{Deng-Guan-Zhang16}, the domains $D_k$ are required to be increasing,
but the same proof can give the result in the above lemma.

\begin{lem}[\cite{Deng-Guan-Zhang16}]\label{lem:squeezing tend 1}
Let $D\subset\mc^m$ be a bounded domain such that the closure $\overline D$ of $D$
admits a Stein open neighborhood basis. Then for any $C^2$ strictly pseudoconvex boundary point
$p$ of $D$, we have $\lim_{z\ra p}s_D(z)=1$.
\end{lem}

We can now give the proof of Theorem \ref{thm-intr:exhauting domain}.
For convenience, we restate it here.

\begin{thm}\label{thm:exhauting domain}
Let $D$ be a bounded strictly pseudoconvex domain in $\mc^n$ with $C^2$ boundary.
If a bounded domain $\Omega\subset\mc^n$ can be exhausted by $D$,
then $\Omega$ must be biholomorphic to $D$ or the unit ball $B^n$.
\end{thm}
\begin{proof}
Let $f_k:D\ra \Omega$ ($k\geq 1$) be a sequence of holomorphic injective maps
satisfying, for any compact set $K\subset \Omega$, there is $N>0$ such that
$K\subset f_k(D)$ for all $k>N$.

Fix a point $w_0\in \Omega$, we consider the inverse images $\{z_k:=f^{-1}_k(w_0);k\geq 1\}$.
We first assume that the set has an accumulation point $z_0\in D$.
Without loss of generality, we may assume $\lim_{k\ra\infty} z_k=z_0$.
Let $g_k:=f^{-1}_k:f_k(D)\ra D$. By Montel theorem,
we may assume $f_k$ converges to $f:D\ra\mc^n$ and $g_k$ converges to $g:\Omega\ra\mc^n$
uniformly on compact sets.
We want to show that $f(D)\subset \Omega$ and $g(\Omega)\subset D$ and
$f\circ g=g\circ f=Id$.
By assumption, we have $f_k(z_k)\ra w_0=f(z_0)$ as $k\ra\infty$.
By Cauchy inequality, we see $J_g(w_0)\neq 0$ and $J_f(z_0)\neq 0$,
where $J_f$ and $J_g$ are the Jacobian of $f$ and $g$.
By Rouch\'e's theorem, $J_f\neq 0$ and $J_g\neq 0$ everywhere.
In particular, $f$ and $g$ are locally biholomorphic.
By the generalized Ruch\'e's theorem for holomorphic maps (see \cite{Llo79}),
we have $f(D)\subset \Omega$ and $g(\Omega)\subset D$.
Then the equalities
$f\circ g =Id_\Omega$ and $g\circ f=Id_D$ follows obviously.
Therefore, in this case, we get that $\Omega$ is biholomorphic to $D$.

We now assume that $\{z_k:=f^{-1}_k(w_0);k\geq 1\}$ has no accumulation point inside $D$.
Then we may assume that $\lim_{k\ra\infty}z_k=p$ for some $p\in\partial D$.
By biholomorphic invariance, we have $s_D(z_k)=s_{f_k(D)}(w_0)$ for all $k$.
By Lemma \ref{lem:stability squeezing}, we have $\lim_k s_{f_k(D)}(w_0)=s_\Omega(w_0)$.
By Lemma \ref{lem:squeezing tend 1}, we have $\lim_k s_D(z_k)=1$ and hence $s_{\Omega}(w_0)=1$.
By Lemma \ref{lem:cha of ball by squeezing}, $\Omega$ is biholomorphic to the unit ball.
\end{proof}

\section{Exhausting domains by homogenous regular domains}
In this section we give the proof of Theorem \ref{thm-intr:exhausted by USD},
which says that a domain is homogenous regular if it can be exhausted by a homogenous regular domain.

Recall that a bounded domain $D$ is called homogenous regular
if the squeezing function $s_D$ of $D$ has a positive lower bound.

We now give the proof of Theorem  Theorem \ref{thm-intr:exhausted by USD}.

\begin{proof}[Proof of Theorem  Theorem \ref{thm-intr:exhausted by USD}]
Let $D$ be a homogenous regular domain.
We assume that $s_D\geq c$ for some constant $c>0$.
Assume that $\Omega$ is a bounded domain that can be exhausted by $D$.
We want to show that $s_\Omega\geq c$ and hence $\Omega$ is also homogenous regular.

Let $f_k:D\ra\Omega$ be holomorphic injective maps such that for any compact set $K\subset\Omega$,
there is $N>0$ such that $K\subset f_k(D)$ for all $k\geq N$.

Let $w\in \Omega$ and set $z_k=f^{-1}(w)$.
By Lemma \ref{lem:stability squeezing}, $s_\Omega(w)=\lim_k s_{f_k(D)}(z_k)$.
By invariance of squeezing functions under biholomorphic transformations,
we have  $s_{f_k(D)}(z_k)=s_D(z_k)$ and hence $s_{f_k(D)}(z_k)\geq c$ for all $k$.
Hence $s_\Omega(w)\geq c$.
Since $w$ is arbitrary, $s_\Omega\geq c$ and hence $\Omega$ is also homogenous regular.
\end{proof}

\section{Boundary estimates of Fridman's invariants}
This section is to prove Theorem \ref{thm-intr:ball charac}.

We first give an estimate of upper bound of the Kobayashi distance
of a bounded domain near a $C^2$ boundary point.

\begin{lem}\label{lem:kobayashi estimate}
Assume $D \subset \mathbb{C}^{n}$ is a bounded domain with $z_0\in D$.
Let $p\in\partial D$ be a $C^2$ boundary point of $D$.
Let $d_D(\cdot,\cdot)$ be the Kobayashi distance on $D$.
Then there exists a constant $C$ such that
\begin{equation}
d_{D}(z_0,z) \leq \log\frac{1}{\delta(z)}+C
\end{equation}
for $p$ in some small neighborhood of $p$,
where $\delta(z)$ is the Euclidean distance from $z$ to $\partial D$.
\end{lem}
\begin{proof}
By some basic results from differential topology,
there is an open neighborhood $U$ such that for any $z\in U\cap D$
there is a unique point $\pi(z)\in U\cap\partial D$ such that $\delta(z)=d(z,\pi(z))$.
Let $r>0$ be sufficiently small and let $M:=\{z\in \overline V\cap D; \delta(z)=r\}$,
where $V\subset U$ is an open neighborhood of $p$ such that $\overline V\subset U$.
Then for any $z\in M$, the 1-dimensional complex disc $\Delta_z$ in the direction $\pi(z)-z$ centered at $z$
and with radius $r$ is contained in $D$.

Now for any $z\in D$ which is sufficiently close to $p$,
we can find a unique $a\in M$ such that $z\in \Delta_{a}$.
By the triangle inequality and decreasing property for Kobayashi distance,
we have
$$d_D(z_0, z)\leq d_D(z_0, a)+d_{\Delta_a}(a,z)= d_D(z_0, a)+\log\frac{2r-\delta(z)}{\delta(z)}.$$
Note that $M$ is compact and $r$ is independent of $z$,
$D(z_0, z)\leq C+\log \frac{1}{\delta(z)}$ for some constant $C>0$.
\end{proof}

We now give the proof of Theorem \ref{thm-intr:ball charac}.
For convenience, we restate it here.

\begin{thm}\label{thm:ball charac}
Let $D\subset\mc^n$ be a bounded domain and $p\in \partial D$ be a $C^2$ boundary point of $D$.
If there is a sequence $z_j\in D\ (j\geq 1)$ converging to $p$ and a sequence of positive numbers
$\epsilon_j\ (j\geq 1)$ converges to $0$ such that $e_D(z_j)>1-\epsilon_j \delta(z_j)$ for all $j$,
then $D$ is biholomorphic to the unit ball,
where $\delta(z)$ denotes the distance between $z$ and $\partial D$.
\end{thm}
\begin{proof}
We will show that $D$ can be exhausted by the unit ball $B^n$.

Let $s_j=1-\epsilon_j \delta(z_j)$ and let
$$r_j=\log\frac{1+s_j}{1-s_j}=\log\frac{2+\epsilon_j \delta(z_j)}{\epsilon_j \delta(z_j)}.$$
By definition of $e_D(z)$, for each $j$, there is a holomorphic injective map
$f_j:B^n\ra D$ such that $f_j(0)=z_j$ and $B^K_D(z_j,r_j)\subset f_j(B^n)$,
where $B^K_D(z_j,r_j)$ is the ball with respect to the Kobayashi metric in $D$
with center $z_j$ and radius $r_j$.

Let $A$ be an arbitrary compact set in $D$.
By Lemma \ref{lem:kobayashi estimate} and the continuity of Kobayashi metric,
there is a constant $C>0$ independent of $z_j$ such that $d_D(A,z_j)\leq \log \frac{1}{\delta(z_j)}+C$ for $j$ sufficiently large,
where $d_D(\cdot,\cdot)$ is the Kobayashi distance on $D$.
This implies that $A\subset B^K_D(z_j,r_j)\subset f_j(B^n)$ for $j$ large enough.
Therefore $\Omega$ can exhausted by the unit ball $B^n$.
By Theorem \ref{thm-intr:exhauting domain}, $D$ is biholomoprhic to $B^n$.
\end{proof}

\maketitle

\end{document}